\documentclass[11pt,a4paper,draft]{amsart}
\usepackage{amsfonts,amsmath,amssymb,amsthm}
\usepackage{times}
\usepackage{color}
\usepackage{pstricks}

\numberwithin{equation}{section}

\DeclareSymbolFont{SY}{U}{psy}{m}{n}
\DeclareMathSymbol{\emptyset}{\mathord}{SY}{'306}

\DeclareMathOperator{\Ran}{Ran} 
 \DeclareMathOperator{\Dom}{Dom}

\DeclareMathOperator{\spec}{spec}

\DeclareMathSymbol{\newtimes}{\mathbin}{SY}{'264}

\newcommand{\eps}{\varepsilon}
\newcommand{\norm}[1]{\|#1\|}

\newcommand{\abs}[1]{|#1|}

\newcommand{\dist}{\mathrm{dist}}

\newcommand{\R}{\mathbb{R}}

\newcommand{\C}{\mathbb{C}}
\newcommand{\Z}{\mathbb{Z}}
\newcommand{\N}{\mathbb{N}}

\newcommand{\EE}{\mathsf{E}}

\newcommand{\ft}{\mathfrak{t}}
\newcommand{\fs}{\mathfrak{s}}

\newcommand{\fx}{\mathfrak{x}}

\newcommand{\cG}{{\mathcal G}}
\newcommand{\cH}{{\mathcal H}}

\newcommand{\cL}{{\mathcal L}}

\newcommand{\cO}{{\mathcal O}}
\newcommand{\cP}{{\mathcal P}}

\newcommand{\cU}{{\mathcal U}}
\newcommand{\cV}{{\mathcal V}}

\newcommand{\cW}{{\mathcal W}}

\newcommand{\ii}{\mathrm{i}}

\newtheorem{theorem}{Theorem}[section]{\bf}{\it}
\newtheorem{proposition}[theorem]{Proposition}{\bf}{\it}
\newtheorem{corollary}[theorem]{Corollary}{\bf}{\it}
{\it}{\rm}
\newtheorem{lemma}[theorem]{Lemma}{\bf}{\it}
\newtheorem{remark}[theorem]{Remark}{\it}{\rm}
\newtheorem{definition}[theorem]{Definition}{\bf}{\it}
{\bf}{\it}
{\bf}{\it}
{\bf}{\it}
\newtheorem{hypothesis}[theorem]{Hypothesis}{\bf}{\it}

\title[Orthogonal Projections and their applications]{Metric properties of the set of orthogonal projections and their applications to operator perturbation theory}

\author[K.\ A.\ Makarov]{Konstantin A.\ Makarov}
\address{K.~A.~Makarov, Department of Mathematics, University of Missouri, Columbia, MO 65211, USA}
\email{makarovk@missouri.edu}

\author[A.\ Seelmann]{Albrecht Seelmann}
\address{A.~Seelmann, FB 08 - Institut f\"{u}r Mathematik,
Johannes Gutenberg-Universit\"{a}t Mainz,
Staudinger Weg 9,
D-55099 Mainz,
Germany}
\email{seelmann@mathematik.uni-mainz.de}

\date{\today}

\copyrightinfo{2009}{A.~Seelmann}

\begin{document}

\begin{abstract}
We prove that the set of orthogonal projections on a Hilbert space
equipped with the length metric is $\frac\pi2$-geodesic. As an
application, we consider the problem of variation of spectral
subspaces for bounded linear self-adjoint operators and obtain a new
estimate on the norm of the difference of two spectral projections
associated with isolated parts of the spectrum of the perturbed and
unpertubed operators, respectively. In particular, recent results by
Kostrykin, Makarov and Motovilov from [Trans.\ Amer.\ Math.\ Soc.,
V. 359, No. 1, 77 -- 89] and [Proc.\ Amer.\ Math.\ Soc., 131, 3469
-- 3476] are sharpened.
\end{abstract}

\maketitle

\section{Introduction}

 The main purpose of this paper is to study metric properties of
the (noncommutative) space $\cP$ of orthogonal projections acting in
a separable  Hilbert space $\cH$ with the emphasis on applications
to the spectral perturbation theory. On the metric space $(\cP, d)$,
where $d$ is the metric introduced by the norm in the space
$\cL(\cH)$ of bounded operators on $\cH$,
$$
d(P,Q)=\norm{P-Q}, \quad P,Q\in \cP,
$$
we introduce the length metric $\rho$, so that the space $(\cP,
\rho)$ becomes a length space, with the distance $\rho$ between two
points  defined as the infimum of the lengths of the paths that join
them.

One of our principle results regarding the global geometry of the space of projections $\cP$ is that the length space $(\cP, \rho)$
is $\frac\pi2$-geodesic. That   means that any two projections $P,Q\in \cP$ with $\rho(P,Q)<\frac\pi2$ can be connected by a geodesic path of length $l=\rho(P,Q)$. Recall that a path $\gamma:[a,b]\to \cP$ is called a geodesic if
$$
\rho(\gamma(t),\gamma(s))=|t-s|,\quad t,s\in[a,b].
$$
 In particular, we prove that the collection of the open unit balls in $(\cP, d)$ coincides with the one of the open balls
of radius $\frac\pi2$ in the length space $(\cP,\rho)$, that is,
$$
\norm{P-Q}<1 \quad \text{iff}\quad \rho(P,Q)<\frac\pi2 \quad \text{for }\quad \quad P,Q\in \cP.
$$

The pairs $(P,Q)$ of orthogonal projections with $\norm{P-Q}<1$ are of special interest.
For instance, such $P$ and $Q$ are unitarily equivalent. Moreover, $\Ran Q$ is a graph subspace of a bounded operator $X: \Ran P \to \Ran P^\perp$ and hence the relative geometry of the subspaces $\Ran P$ and $\Ran Q$
can efficiently be studied by  using standard tools of the geometric  perturbation theory.
The key role in our study of the relative geometry of the  graph subspaces $ \Ran P$ and $\Ran Q$ with $ \rho(P,Q)<\frac\pi2 $ is played by the  operator angle  $\Theta$, a  self-adjoint operator that can be introduced via the operator $X$
by the functional calculus
$$
\Theta=\arctan (X^*X)^{1/2}.
$$
Using the concept of the  operator angle we  show  that the length metric $\rho$ is
locally characterized by the norm of $\Theta$:
$$
\rho(P,Q)=\norm{\Theta}
\quad\text{if}\quad \rho(P,Q)<\frac\pi2\,.
$$
 Using  the characterization of the length metric as the infimum of the arc lengths and the well known relation $\norm{\Theta}=\arcsin \norm{P-Q}$,
we prove the following sharp  inequality
$$
\arcsin \norm{P-Q}\le \int_a^b\norm{\dot \gamma(t)}dt
$$
relating the norm of the difference of orthogonal projections and the arc length of a smooth path $\gamma:[a,b]\to\cP$ joining them.

As the first application of our geometric study of the space $\cP$
to the spectral perturbation theory, we consider a smooth
self-adjoint path of bounded operators $B_t$ each having two
disjoint spectral components. Given that the two families
$\{\omega_t\}_{t\in  I}$ and $\{\Omega_t\}_{t\in  I}$
 of spectral
components depend upper semicontinuously on the parameter, we prove
the following inequality
\[
 \arcsin(\norm{P_t-P_0}) \le \frac\pi2 \int_0^t \frac{\dot B_\tau}{\dist(\omega_\tau,\Omega_\tau)}\, d\tau\,,\quad t\in I\,,
\]
where $P_t$ denotes the spectral projection for $B_t$ associated
with the  spectral component $\{\omega_t\}_{t\in  I}$.

As an immediate consequence, we obtain  new estimates in the
subspace perturbation problem recently considered in   \cite{Kos1}
and \cite{Kos3}.

The paper is organized as follows.

In Section 2 we  start with  recalling basic facts on orthogonal
projections and prove an  important technical  result (see,
Corollary \ref{tan},  The Four Projections Lemma).

In Section \ref{sec:smthPaths} we deal with  smooth paths of
projections. As a key result we relate the norm of the difference of
the two endpoints of a smooth path and the corresponding arc length
(see, Lemma \ref{lem:resPiecSmth}, the Arcsine Law for smooth
paths).

In Section \ref{sec:geodMetrSp} we provide a characterization of the
local geodesic structure of the length space $(\cP,\rho)$ and prove
that the metric space  $(\cP,\rho)$ is $\pi/2$-geodesic. In
particular, we generalize the Arcsine Law  from Section
\ref{sec:smthPaths} to the case of continuous paths.

In  Section \ref{sec:appl} we  apply the results from the preceding
sections to the problem of variation of spectral subspaces including
some discussions about  the optimality of the obtained estimates.

In Section \ref{sec:newEstsubPertProb} we obtain new estimates in
the subspace perturbation problem  sharpening recent results  from
\cite{Kos1} and  \cite{Kos3}.

\textbf{Acknowledgments. }
The authors are grateful to Vadim Kostrykin for stimulating discussions.
A.S. would like to express his gratitude to his scientific advisor Vadim Kostrykin
for introducing him to the field of the research and continuous support.
K.A.M.\ is indebted to the Institute for Mathematics for its kind hospitality during his stay at the
Johannes Gutenberg-Universit\"at Mainz in the Summer of 2009 and 2010. The work of K.A.M.\ has been supported in part by the
Deutsche Forschungsgesellschaft and by the Inneruniversit\"aren Forschungsf\"orderung of the Johannes Gutenberg-Universit\"at Mainz.



\section{Preliminaries}\label{sec:Preliminaries}
We start with recalling some important
facts on the representation for the range of an orthogonal projection as a graph subspace associated with
the range of another orthogonal projection. For the proofs the reader is referred to the work
\cite{Kos2}.

Let $P$ and $Q$ be orthogonal projections in the Hilbert space $\cH$, where we will tacitly understand $\cH$
to be separable throughout this paper.
It is well known that the inequality $\|P-Q\|<1$ holds true if
and only if $\Ran Q$ is a graph of a bounded operator  $X\in\cL(\Ran P,\Ran P^\bot)$, $P^\perp:=I_\cH-P$, that is,
\begin{equation*}
    \Ran Q=\cG(X):=\cG(\Ran P,X):=\{x_0\oplus Xx_0 \mid x_0\in\Ran P\}\,.
\end{equation*}
In this case  the projection $Q$ has the following representation as a block operator matrix with
respect to the orthogonal decomposition $\cH=\Ran P\oplus\Ran P^\bot$:
\begin{equation}\label{eq:SpectralProjector}
    Q=\begin{pmatrix} (I_{\cH_0} + X^*X)^{-1} & (I_{\cH_0}+X^*X)^{-1}X^*\\ X(I_{\cH_0}+X^*X)^{-1} &
    X(I_{\cH_0}+X^*X)^{-1}X^* \end{pmatrix},
\end{equation}
where $\cH_0:=\Ran P$ (cf. Remark 3.6 in \cite{Kos2}).

The knowledge of the angular operator $X$ and/or  the operator angle $\Theta$ (see, e.g., \cite{Kos2} for a discussion of this concept)
between the subspaces $\Ran P$ and $\Ran Q$ given by
$$
 \Theta=\arctan \sqrt{X^*X},
$$provides complete information on relative geometry of the subspaces $\Ran P$ and $\Ran Q$. In particular,
\begin{equation}\label{eq:NormSol}
    \|X\|= \frac{\|P-Q\|}{\sqrt{1-\|P-Q\|^2}}=\tan \|\Theta\|
\end{equation}
and
\begin{equation}\label{eq:NormProjDiff}
    \|P-Q\| = \frac{\|X\|}{\sqrt{1+\|X\|^2}}=\sin \|\Theta\|
\end{equation}
(see, e.g.,  Corollary 3.4 in \cite{Kos2}).

Moreover, in this case, the orthogonal projections $P$ and $Q$ are unitarily equivalent. In particular,
\begin{equation*}
P=U^*QU,
\end{equation*}
where $U$ is given by the following  unitary block operator matrix
\begin{equation}\label{eq:UnitaryTransform}
    U=\begin{pmatrix}
        (I_{\cH_0} + X^*X)^{-1/2} & -X^*(I_{\cH_1} + XX^*)^{-1/2}\\
            X(I_{\cH_0} + X^*X)^{-1/2} & (I_{\cH_1} + XX^*)^{-1/2}
      \end{pmatrix},\quad \cH_1:=\Ran P^\perp\,.
\end{equation}

Our next result is a purely algebraic observation  the proof of
which  requires nothing but straightforward multiplication of
several operator matrices and hence will be omitted.


\begin{lemma}[Four projections lemma]\label{tri} Assume that $P$, $Q_1$, and $Q_2$  are orthogonal projections such that
$$\|P-Q_j\|<1, \quad j=1,2,
$$
and therefore
$$
\Ran Q_j=\cG (X_j), \quad j=1,2,
$$
for some angular operators $X_j\in \cL(\Ran P, \Ran P^\perp)$. Let $U_1$ be
the corresponding unitary operator from \eqref{eq:UnitaryTransform} such that $P=U_1^*Q_1U_1$, that is
$$U_1=\begin{pmatrix}
        (I_{\cH_0} + X_1^*X_1)^{-1/2} & -X_1^*(I_{\cH_1} + X_1X_1^*)^{-1/2}\\
            X_1(I_{\cH_0} + X_1^*X_1)^{-1/2} & (I_{\cH_1} + X_1X_1^*)^{-1/2}
      \end{pmatrix}
$$
with $\cH_0=\Ran P$ and $\cH_1=\Ran P^\perp$.
Then the orthogonal projection $Q$ given by
$$
Q=U_1^*Q_2U_1,
$$
admits the factorization
$$
Q=A^{-1/2}BCB^*A^{-1/2},
$$
where  $A\in \cL(\cH)$, $B\in \cL(\Ran P, \Ran P^\perp)$ and $C\in \cL(\Ran P)$ are $2\times 2$, $2\times 1$ and $1\times 1$  block operator matrices  (with respect to the orthogonal decomposition $\cH=\Ran P\oplus \Ran P^\perp$) respectively,
given by
\begin{align}
A&=\begin{pmatrix}I_{\cH_0}+X_1^*X_1&0\\
0&I_{\cH_1}+X_1X_1^*
\end{pmatrix}\label{a},
\\
B&=\begin{pmatrix}
I_{\cH_0}+X_1^*X_2\\
X_2-X_1
\end{pmatrix},\label{b}
\\
C&=(I_{\cH_0}+X_2^*X_2)^{-1}\label{c}.
\end{align}
\end{lemma}

The last statement  of this  preliminary section allows one to
compare the angular operators $X_1$ and $X_2$ associated with the
graph subspaces $\Ran Q_1$ and $\Ran Q_2$ referred to in Lemma
\ref{tri}. As a result,  one obtains the following ``angle
addition'' formula.


\begin{corollary}\label{tan} Suppose in addition to the assumptions of Lemma \ref{tri} that  the range of the orthogonal  projection $Q$
is a graph subspace with respect to the decomposition $\cH=\Ran P\oplus \Ran P^\perp=:\cH_0\oplus\cH_1$, and therefore
$$
\Ran Q= \cG(Z) \quad \text{ for some }\,\,\,Z\in \cL(\Ran P,\Ran P^\perp).
$$
Moreover, assume that the operator $I_{\cH_0}+X_2^*X_1\in \cL(\Ran P)$ is of full range, that is,
$$
\Ran ( I_{\cH_0}+X_2^*X_1)=\Ran P.
$$

Then
\begin{equation}\label{trig}
X_2-X_1=(I_{\cH_1}+X_1X_1^*)^{1/2}Z(I_{\cH_0}+X_1^*X_1)^{-1/2}(I_{\cH_0}+X_1^*X_2).
\end{equation}
\end{corollary}
\begin{proof} From the definition of the angular operator $Z$, i.e. $\Ran Q=\cG(Z)$, it follows that
\begin{equation}\label{r0}
P^\perp Q=Z P Q.
\end{equation}
Recall that
by Lemma \ref{tri},
$$
Q=A^{-1/2}BCB^*A^{-1/2},
$$
where the operators $A$, $B$, and $C$ are given by \eqref{a}-\eqref{c}.
In particular,
$$
CB^*A^{-1/2}|_{\Ran P}=(I_{\cH_0}+X_2^*X_2)^{-1}( I_{\cH_0}+X_2^*X_1)(I_{\cH_0}+X_1^*X_1)^{-1/2}\,.
$$
By hypothesis, the operator $(I_{\cH_0}+X_2^*X_1)$ is of full range, so is $CB^*A^{-1/2}|_{\Ran P}$.
Therefore,  \eqref{r0} implies the equality
$$
P^\perp A^{-1/2}B=Z P A^{-1/2}B.
$$

Taking into account representations \eqref{a} and \eqref{b}, one computes
\begin{equation}\label{r1}
P^\perp A^{-1/2}B=(I_{\cH_1}+X_1X_1^*)^{-1/2}(X_2-X_1)
\end{equation}
and
\begin{equation}\label{r2}
PA^{-1/2}B=(I_{\cH_0}+X_1^*X_1)^{-1/2}(I_{\cH_0}+X_1^*X_2).
\end{equation}
Combining \eqref{r0}, \eqref{r1},  and \eqref{r2}, one concludes that
\begin{equation}\label{pochti}
(I_{\cH_1}+X_1X_1^*)^{-1/2}(X_2-X_1)=Z(I_{\cH_0}+X_1^*X_1)^{-1/2}(I_{\cH_0}+X_1^*X_2)
\end{equation}
and the claim follows by multiplying both sides of \eqref{pochti} by the operator $(I_{\cH_1}+X_1X_1^*)^{1/2}$ from the left.
\end{proof}
\begin{remark}  Representation \eqref{trig} relating the angular operators $X_1$, $X_2$ and $Z$
is a non-commutative variant of  the ``angle addition'' formula
\begin{equation}\label{add}
\tan \Theta_2-\tan \Theta_1 = \tan ( \Theta_2- \Theta_1) \cdot \left(1+\tan \Theta_1\tan \Theta_2\right)\,.
\end{equation}
To ``justify'' this observation, consider an example of a space of dimension 2 and rank 1 orthogonal projections $Q_1$ and $Q_2$ whose ranges are lines of inclinations $\Theta_1$ and $\Theta_2$, respectively.
Then the operator $Q_2$ in the new coordinate system
\begin{align*}
x'&=\cos \Theta_1x +\sin \Theta_1 y\\
y'&=-\sin \Theta_1x +\cos \Theta_1 y
\end{align*}
turns out to be a rank 1 projection  $Q$ whose range is a line of the slope $\tan(\Theta_2-\Theta_1)$.
Since the angular operators $X_1$, $X_2$, and $Z$  play the role of the slope of the line, in the case in question the angle addition formula \eqref{add} is equivalent to the relation \eqref{trig}.

\end{remark}


\section{Smooth Paths of projections}\label{sec:smthPaths}

Throughout this section we consider the set of orthogonal
projections $\cP$ in a Hilbert space $\cH$,
\[
 \cP = \{P\in\cL(\cH) \mid P=P^*=P^2\}\,,
\]
 as a metric space
with respect to the metric $d$ induced by the operator norm on
$\cL(\cH)$.

Recall, that a piecewise $C^1$-smooth path is a mapping $\gamma\colon[a,b]\to\cP$ such that there is a partition $a=t_0<\dots<t_n=b$ and
$\gamma|_{[t_j,t_{j+1}]}$ is $C^1$-smooth for all $j\in\{0,\dots,n-1\}$. In particular, such paths are continuous.
\begin{hypothesis}\label{main}
Assume that $\gamma\colon [a,b]\to\cP$ is a piecewise $C^1$-smooth path of orthogonal projections.
Suppose, in addition, that $\gamma(t)$ for all $t\in[a,b]$ is a graph subspace associated with a bounded operator
with respect to the orthogonal decomposition $\cH=\Ran \gamma(a)
\oplus \Ran \gamma(a)^\perp$, that is
$$
\Ran \gamma(t)=\cG(X_t) \quad \text{ for some angular operator } \,\,X_t\in\cL(\Ran \gamma(a), \Ran \gamma(a)^\perp),\quad t\in [a,b].
$$
\end{hypothesis}


Our first result in this section shows, that smoothness of the path
of projections implies smoothness of the corresponding angular
operators in the graph subspace representation. The exact statement
is as follows.

\begin{lemma}\label{lem:graphOpSmooth}
Assume Hypothesis \ref{main}.
If $I\subset[a,b]$ is an interval, such that $\gamma|_I$ is $C^1$-smooth, then
$I\ni t\mapsto X_t$ is also $C^1$-smooth.
In particular, the path $[a,b]\ni t \mapsto X_t$ is piecewise $C^1$-smooth.
\begin{proof}
    Let $\cH_0=\Ran \gamma(a)$ and $\cH_1=\Ran \gamma(a)^\bot$, where $\gamma(a)^\bot=I_\cH-\gamma(a)$.
    Introduce piecewise $C^1$-smooth  families of bounded  operators given by
     $$T _t:= \gamma(a)^\bot \gamma(t) \gamma(a)\quad \text{and} \quad S_t := \gamma(a) \gamma(t) \gamma(a) ,\quad t\in [a,b].$$
      Denote by $R_t$ the following operator matrix
    with respect to the decomposition $\cH=\cH_0\oplus \cH_1$,
\begin{equation}
R_t := \begin{pmatrix} I_{\cH_0}+X_t ^*X_t& 0\\ 0 & 0\end{pmatrix},\quad t\in [a,b].
\end{equation}

Using  \eqref{eq:SpectralProjector}, one obtains that for each $t\in [a,b]$
\begin{equation}
    \gamma(t)=\begin{pmatrix} (I_{\cH_0} + X_t^*X_t)^{-1} & (I_{\cH_0}+X_t^*X_t)^{-1}X_t^*\\ X_t(I_{\cH_0}+X_t^*X_t)^{-1} &
    X_t(I_{\cH_0}+X_t^*X_t)^{-1}X_t^* \end{pmatrix},
\end{equation}
and a simple computation shows that the operators $T_t$ and $S_t$ can be represented as the following operator matrices with respect to the decomposition $\cH=\cH_0\oplus \cH_1$:
    \[
        \begin{split}
            T_t &= \gamma(a)^\bot \gamma(t) \gamma(a) =
                \begin{pmatrix} 0 & 0\\ X_t \bigl(I_{\cH_0}+X_t ^*X_t \bigr)^{-1} & 0
                \end{pmatrix},\\[0.1cm]
            S_t &= \gamma(a) \gamma(t) \gamma(a) =
                \begin{pmatrix} \bigl(I_{\cH_0}+X_t ^*X_t\bigr)^{-1} & 0\\ 0 & 0\end{pmatrix},
            \quad t\in [a,b]
        \end{split}
\]

    Now it is easy to see that the ``resolvent identity"
    \begin{equation}\label{rrss}
        R_s- R_t= R_t \left (S_t - S_s\right ) R_s\,,\quad s,t\in [a,b].
    \end{equation}
    holds.

        The norm estimate
    $$
    \|R_t\|\le 1+\|X_t\|^2
    $$
    combined with the identity (see \eqref{eq:NormSol})
    $$
    \|X_t\|=\frac{\|\gamma(t)-\gamma(a)\|}{\sqrt{1-\|\gamma(t)-\gamma(a)\|^2}}
    $$
yields the inequality
\begin{equation}\label{ozen}
    \|R_t\|\le\frac{1}{1-\|\gamma(t)-\gamma(a)\|^2}\,, \quad t\in [a,b].
    \end{equation}
Since by hypothesis $\Ran \gamma(t)$ is a graph of a bounded operator, one gets that
$\|\gamma(t)-\gamma(a)\|<1$ for all $t\in [a,b]$. Due to the continuity of the path  $[a,b] \ni t\mapsto \gamma(t)$, from \eqref{ozen} one concludes that for any $t_0\in [a,b] $ there exists a neighborhood $U_{t_0}$ of the point $t_0$   such that the function
    $U_{t_0}\ni t\mapsto \|R_t\|$ is uniformly bounded.
Taking this observation into account and recalling that the family $S_t$ is piecewise differentiable,
 from the representation \eqref{rrss} it follows that the family $R_t$ is also piecewise differentiable with
    \begin{equation}\label{diff}
    \dot    R_t=
        -R_t \dot S_t R_t\,, \quad t\in I\,,
    \end{equation}
    where $I\subset[a,b]$ is any interval such that $\gamma|_I$ is $C^1$-smooth.
    Since $I \ni t \mapsto \dot S_t$ is a continuous path, from \eqref{diff} it follows that
      $I \ni t\mapsto R_t$ is a $C^1$-smooth path. It remains to observe that
    \[
        \begin{pmatrix} 0 & 0\\ X_t & 0\end{pmatrix} -\begin{pmatrix} 0 & 0\\ X_s& 0
        \end{pmatrix}=T_t R_t - T_s R_s\,,\quad s,t\in I,
    \]
    to conclude that $I \ni t\mapsto X_t $ is a $C^1$-smooth path with
    \[
        \begin{pmatrix} 0 & 0\\ \dot X_t & 0\end{pmatrix} = \dot T_tR_t+ T_t \dot R_t\,, \quad t\in I.
    \]
\end{proof}
\end{lemma}


Our next result forms in fact the core of our considerations for it
relates the evolution of the path of angular operators and the
evolution of the corresponding path of orthogonal projections. It
justifies the following principle: The speed of rotation of the
subspaces $\Ran \gamma(t)$ along a path $[a,b]\ni t \to \gamma(t)$
does not exceed the speed on the path.

\begin{lemma}\label{mmss} Assume Hypothesis \ref{main}.
Let $I\subset[a,b]$ be an interval, such that $\gamma|_I$ is $C^1$-smooth.
Then the estimate
\begin{equation}\label{mosa}
\|\dot X_t\|\le \left(1+\|X_t\|^2\right)\|\dot \gamma(t)\|
\end{equation}
holds for all $t\in I$.
\end{lemma}


\begin{proof} Since, by hypothesis, $\|\gamma(t)-\gamma(a)\|<1$, $t\in I$, the projections $\gamma(t)$ and $\gamma(a)$ are unitarily equivalent. In particular,
\begin{equation}\label{eq:unitEquiv}
\gamma(a)=U_t^*\gamma(t)U_t, \quad t\in I.
\end{equation}
where  the family of unitary operators $U_t$, $t\in I$, is given by \eqref{eq:UnitaryTransform} accordingly.

Fix an $s\in I$ and introduce
 the family of orthogonal projections
\begin{equation}\label{eq:famOrthoProj}
Q_t=U_t^*\gamma(s)U_t, \quad t\in I.
\end{equation}
 Due to the continuity of the path $ I\ni t\mapsto \gamma(t)$, there exists a neighborhood
$\cV\subset I$ of the point $s$, such that
\begin{equation}\label{eq:projLeOne}
\|\gamma(t)-\gamma(s)\| < 1, \quad t\in \cV.
\end{equation}
Since by \eqref{eq:unitEquiv} and \eqref{eq:famOrthoProj}
\begin{equation}\label{eq:altProjEst}
 \|Q_t-\gamma(a)\|= \|U_t^*\gamma(s)U_t-U_t^*\gamma(t)U_t\|=\|\gamma(s)-\gamma(t)\|\,,
\end{equation}
from \eqref{eq:projLeOne} it follows that
$$
\|Q_{t}-\gamma(a)\|<1\,, \quad t\in \cV\,.
$$
Therefore, $\Ran Q_{t}$ is a graph subspace with respect to the
decomposition $\cH=\Ran \gamma(a)\oplus \Ran \gamma(a)^\perp$, that
is
$$
\Ran Q_{t}=\cG(Y_{t})\quad\text{ for some}\quad Y_{t}\in \cL(\Ran \gamma(a), \Ran \gamma(a)^\perp)\,.
$$

Next, one observes that the operator $I_{\cH_0}+X_s^*X_s$ has a bounded inverse and therefore
$I_{\cH_0}+X_t^*X_s$ has a bounded inverse as well for all $t$ from   in a possibly  smaller neighborhood $\tilde \cV\subset \cV$ of the point $s$.
In particular, the operator $I_{\cH_0}+X_t^*X_s$ is of full range for all $t\in \tilde \cV$ and
one can apply Corollary  \ref{tan} to get the representation
$$
X_t-X_s=(I_{\cH_0}+X_sX_s^*)^{1/2}Y_t(I_{\cH_0}+X_s^*X_s)^{-1/2}(I_{\cH_0}+X_s^*X_t)\,,\quad t\in \tilde \cV,
$$
and hence
\begin{equation}\label{intan}
\|X_t-X_s\|\le \| (I+X_sX_s^*)^{1/2}\|\cdot \|Y_t\|
\cdot \|(I+X_s^*X_s)^{-1/2}(I+X_s^*X_t)\|\,,\quad t\in \tilde \cV\,.
\end{equation}
Since by \eqref{eq:NormSol}
$$
\|Y_t\|=\frac{\|Q_t-\gamma(a)\|}{\sqrt{1-\|Q_t-\gamma(a)\|^2}}
$$
from \eqref{eq:altProjEst}
one obtains that
$$
\lim_{t\to s}\frac{\|Y_t\|}{t-s}=\lim_{t\to s}\frac{\|\gamma(t)-\gamma(s)\|}{t-s}=\|\dot \gamma(s)\|,
$$
for $ I\ni t\mapsto \gamma(t)$ is a $C^1$-smooth path. Since by Lemma \ref{lem:graphOpSmooth}
$ I\ni t\mapsto X_t$ is also a $C^1$-smooth path, from inequality \eqref{intan} one gets the estimate
\begin{align*}
\|\dot X_s\|&\le \| (I+X_sX_s^*)^{1/2}\|\cdot \|\dot \gamma(s)\|
\cdot \|(I+X_s^*X_s)^{1/2}\|
\\
&=(1+\|X_s\|^2) \|\cdot \|\dot \gamma(s)\|.
\end{align*}
 Since  the reference  point $s\in I$ has been  chosen arbitrarily, one proves the inequality \eqref{mosa}.
\end{proof}

Using the information  about the evolution of the angular operators
provided by  Lemma \ref{mmss}, we are now able to estimate the
variation of the corresponding orthogonal projections.


\begin{lemma}[The Arcsine Law]\label{lem:resPiecSmth}
 Let $\gamma\colon[a,b]\to\cP$ be a piecewise $C^1$-smooth path. Then
    \begin{equation}\label{eq:resPiecwSmth}
     \arcsin(\norm{\gamma(b)-\gamma(a)}) \le l_R(\gamma)\,,
    \end{equation}
    where
    \[
     l_R(\gamma) = \int _a^b \norm{\dot\gamma(t)}\,d t
    \]
    is the Riemannian length of the path $\gamma$.


    \begin{proof}
     Since $\norm{\gamma(b)-\gamma(a)}\le 1$, we may assume $l_R(\gamma)<\frac{\pi}{2}$.

     Let $a=t_0<\dots <t_n=b$ be a partition such that $\gamma|_{[t_j,t_{j+1}]}$ is $C^1$-smooth.
     Set
     \begin{equation}\label{eq:critic}
        T := \sup\left\{t\in[a,b] \,\middle|\, \norm{\gamma(t')-\gamma(a)}<1\ \text{ for all }\ t'\in[a,t)\right\}\,.
     \end{equation}
     Clearly, $T>a$, for $\gamma$ is continuous. Since
     \[
        \norm{\gamma(t)-\gamma(a)} < 1 \quad \text{ for all }\ t\in [a,T)\,,
     \]
     the range of $\gamma(t)$ is a graph subspace with respect to the decomposition
     $\cH=\Ran \gamma(a)\oplus \Ran \gamma(a)^\perp$ and therefore
     \[
        \Ran \gamma(t)=\cG( X_t) \quad \text{ for some }\ X_t\in \cL( \Ran \gamma(a), \Ran \gamma(a)^\perp)\,, \quad
        t\in [a,T)\,.
     \]
        Due to Lemma \ref{mmss}, we have 
     \[
        \|\dot X_t\| \le ( 1+\|X_t\|^2)\|\dot \gamma(t)\|
     \]
     for all $t\in(t_j,t_{j+1})$, $j=0,\dots,n-1$, as long as $t<T$. For arbitrary $t\in[a,T)$ there is a unique
        $k\in\{0,\dots,n-1\}$
     such that $t\in[t_k,t_{k+1})$. We obtain
     \begin{equation}\label{eq:int}
      \begin{aligned}
         \norm{X_t}
         &= \norm{X_t-X_a} \le \norm{X_t-X_{t_k}} + \sum_{j=0}^{k-1}\norm{X_{t_{j+1}}-X_{t_j}}\\
         &\le \int_{t_k}^t \|\dot X_\tau\|\, d\tau + \sum_{j=0}^{k-1}\int_{t_j}^{t_{j+1}} \|\dot X_\tau\|\, d\tau\\
         &\le \int_{t_k}^t ( 1+\|X_\tau\|^2)\|\dot \gamma(\tau)\| \, d\tau + \sum_{j=0}^{k-1}\int_{t_j}^{t_{j+1}}
            ( 1+\|X_\tau\|^2)\|\dot \gamma(\tau)\| \, d\tau\\
            &= \int_a^t ( 1+\|X_\tau\|^2)\|\dot \gamma(\tau)\|\, d\tau
        \end{aligned}
     \end{equation}
     for $t\in[a,T)$. Denoting the right hand side of \eqref{eq:int} by $F(t)$, i.e.\
     \begin{equation}\label{eq:feq}
        F(t) = \int_0^t ( 1+\|X_\tau\|^2)\|\dot \gamma(\tau)\|\, d\tau\,,
     \end{equation}
     one concludes that
     \begin{equation}\label{eq:xt}
        \norm{X_t} \le F(t)\,,\ t\in[a,T)\,,
     \end{equation}
     and hence
     \begin{equation}\label{eq:difineq}
        F'(t) = ( 1+\|X_t\|^2)\|\dot \gamma(t)\| \le (1+F^2(t))\|\dot \gamma(t)\|
     \end{equation}
     for $t\in[a,T)$ except for the finitely many points $t_j$.
     Since by assumption $l_R(\gamma)<\frac\pi2$, one can solve the differential inequality \eqref{eq:difineq}
     on every sub-interval of $[a,T)$ where $F$ is $C^1$-smooth. For $t\in[t_k,t_{k+1}]$, $t<T$, one obtains
     \[
        \begin{aligned}
         \arctan F(t) &= \arctan F(t) - \arctan F(a)\\
         & =  \arctan F(t) - \arctan F(t_k) + \sum_{j=0}^{k-1} \bigl( \arctan F(t_{j+1})-\arctan F(t_j) \bigr)\\
         &\le \int_{t_k}^t \|\dot \gamma(\tau)\|\, d\tau + \sum_{j=0}^{k-1} \int_{t_j}^{t_{j+1}} \|\dot \gamma(\tau)\|\,
            d\tau = \int_{a}^{t} \|\dot \gamma(\tau)\|\, d\tau\,.
        \end{aligned}
     \]
     Together with \eqref{eq:xt} this yields the bound
     \begin{equation}\label{eq:atan}
        \arctan \|X_t\| \le \int_a^t \|\dot \gamma(\tau)\|\, d\tau\le l_R(\gamma)<\frac\pi2\,,\quad t\in[a,T)\,.
     \end{equation}
     Since
     \[
        \arcsin\bigl(\|\gamma(t)-\gamma(a)\|\bigr)= \arctan \left (\|X_t\|
        \right )\,,\quad t \in [a,T)\,,
     \]
     from \eqref{eq:atan} one gets the estimate
     \[
        \arcsin\bigl(\|\gamma(t)-\gamma(a)\|\bigr)\le \int_a^t\|\dot \gamma(\tau)\|\,d\tau\,,
        \quad t\in [a,T)\,,
     \]
     and hence, by continuity,
     \[
      \arcsin\bigl(\|\gamma(t)-\gamma(a)\|\bigr)\le \int_0^t\|\dot \gamma(\tau)\|\,d\tau\le l_R(\gamma)<\frac\pi2\,,
      \quad t\in [a,T]\,.
     \]
     In particular it is $T=b$ by definition of $T$ in \eqref{eq:critic}, which proves \eqref{eq:resPiecwSmth}.
    \end{proof}%
\end{lemma}

The next lemma shows that the inequality of Lemma
\ref{lem:resPiecSmth} is sharp. In   particular,  it states that
given orthogonal projections $P$ and $Q$ with $\norm{P-Q}<1$, one
can construct a  $C^1$-smooth path of minimal length, a geodesic,
among all ($C^1$-smooth) paths connecting $P$ and $Q$. It will turn out
later, that this same path is of minimal lenght even among all continuous
paths connecting $P$ and $Q$.

\begin{lemma}\label{lem:smthConn}
 Let $P,Q\in\cP$ with $\norm{P-Q}<1$. Then there exists a $C^1$-smooth path $\gamma\colon[0,l]\to\cP$
    connecting $P$ and $Q$ such that
    \[
     \norm{\dot\gamma(t)} = 1\,,\ t\in[0,l]\,,
    \]
    where
    \[
     l = \arcsin(\norm{P-Q}) < \frac\pi2\,.
    \]
    \begin{proof}
     Since $\norm{P-Q}<1$, the range of $Q$ is a graph subspace with respect to
    $\Ran P$, i.e. $\Ran Q=\cG(\Ran P,X)$ for some $X\in\cL(\Ran P,\Ran P^\perp)$. Without loss of generality one can assume
        that the pair $(P,Q)$ is generic, that is,
        \[
         \Ran P \cap \Ran Q = \Ran P^\perp \cap \Ran Q^\perp = \{0\}
        \]
        and hence one can write (see \cite[Theorem 2.2]{Kos3})
        \[
         P = \begin{pmatrix} I_{\Ran P} & 0\\ 0 & 0\end{pmatrix},\quad
            Q = \cW^* \begin{pmatrix} \cos^2\Theta & \sin\Theta\cos\Theta\\ \sin\Theta\cos\Theta & \sin^2\Theta\end{pmatrix}\cW
        \]
        with respect to the decompostion $\cH=\Ran P \oplus \Ran P^\perp$, where $\Theta=\arctan\sqrt{X^*X}$ is the corresponding
        operator angle and $\cW$ is a unitary operator. In particular,
        \[
            l:=\norm{\Theta}=\arcsin(\norm{P-Q})<\frac\pi2\,.
        \]
        Introduce the $C^1$-smooth path $\gamma\colon[0,l]\to\cP$ connecting $P$ to $Q$ by the following family of block operator
        matrices with respect to the decomposition $\cH=\Ran P \oplus \Ran P^\perp$:
        \[
         \gamma(t) = \cW^* \begin{pmatrix} \cos^2(\Theta\frac tl) & \sin(\Theta \frac tl)\cos(\Theta\frac tl)\\ \sin(\Theta\frac tl)
            \cos(\Theta\frac tl)&
            \sin^2(\Theta\frac tl)\end{pmatrix}\cW\,,\quad t\in[0,l]\,.
        \]
        It remains to observe that
        \[
         \begin{aligned}
          \dot\gamma(t)
                &= \frac 1l \cW^* \begin{pmatrix} -\Theta\sin(2\Theta\frac tl) & \Theta\cos(2\Theta\frac tl)\\
                            \Theta\cos(2\Theta\frac tl) & \Theta\sin(2\Theta\frac tl)\end{pmatrix}\cW\\
                &= \frac 1l\cW^* \begin{pmatrix} \Theta^{1/2} & 0\\ 0 & \Theta^{1/2}\end{pmatrix} J
                            \begin{pmatrix} \Theta^{1/2} & 0\\ 0 & \Theta^{1/2}\end{pmatrix}\cW\,,\quad t\in[0,l]\,,
         \end{aligned}
        \]
        where $J$ is a self-adjoint involution, $J^2=I$, given by
        \[
         J = \begin{pmatrix} -\sin(2\Theta\frac tl) & \cos(2\Theta\frac tl)\\
                            \cos(2\Theta\frac tl) & \sin(2\Theta\frac tl)\end{pmatrix}\,,\ t\in[0,l]\,.
        \]
        Therefore,
        \[
         \norm{\dot\gamma(t)} = \frac{\norm{\Theta}}{l}=1\,,\ t\in[0,l]\,,
        \]
        which completes the proof.
    \end{proof}%
\end{lemma}


\section{The space $\cP$ as a local geodesic metric space}\label{sec:geodMetrSp}

The main goal of this section is to study  metric properties of the
space $\cP$ considered as a length space.

Recall necessary definitions. Given a continuous path
$\gamma\colon[a,b]\to\cP$, its length $l(\gamma)$ is defined by
\[
 l(\gamma) = \sup\left\{ \sum_{j=0}^{n-1} \norm{\gamma(t_{j+1})-\gamma(t_j)}\,\middle|\, n\in\N,\,
     a=t_0 < \dots < t_n=b\right\}\,.
\]
Recall that a path is called rectifiable if its length is finite.

On $\cP$ introduce a length or inner (pseudo-)metric $\rho$ given by the formula
\[
 \rho(P,Q) = \text{infimum of length of rectifiable paths } \gamma \text{ from } P \text{ to } Q\,.
\]
If there are no such paths then set $\rho(P,Q)=\infty$.

It is well known that the pseudometric $\rho$ is actually a metric
and therefore $(\cP,\rho)$ is a well defined metric space (cf.,
\cite[Proposition 3.2]{Bri}, the length space.

There is another way of introducing the inner metric via a ``Riemannian'' arc length of piecewise differentiable paths $\gamma  :[a,b]\to \cP$
from $P$ to $Q$ by
\[
 \rho_R(P,Q) = \text{infimum of } \int_a^b \norm{\dot\gamma(t)}\, d t\,,
\]
and $\rho_R(P,Q)=\infty$ if there are no such paths.

The following lemma shows that the metric spaces $(\cP, \rho)$ and $(\cP, \rho_R)$ coincide.


\begin{lemma}\label{lem:uniApprox}
 Let $\gamma\colon[a,b]\to\cP$ be a continuous path. Then there is a sequence $(\gamma_n)$ of piecewise $C^1$-smooth paths
    $\gamma_n\colon[a,b]\to \cP$, each having the same endpoints as $\gamma$, such that $\gamma_n$ converges uniformly to
    $\gamma$ and its length $l(\gamma_n)$ converges to $l(\gamma)$. In particular, the inner metric $\rho$ and the Riemannian
    pseudometric $\rho_R$ on $\cP$ coincide.
\end{lemma}

    \begin{proof}
     Since $\gamma\colon[a,b]\to\cP$ is uniformly continuous, we can choose for each $n\in\N$ some $N(n)\in\N$ and a partition
        $a=t_0^{(n)}<\dots<t_{N(n)}^{(n)}=b$ such that
        \begin{equation}\label{eq:partition}
         \norm{\gamma(t)-\gamma(s)}<\frac{1}{n}
        \end{equation}
        for all $t,s\in[t_j^{(n)},t_{j+1}^{(n)}]$,
        $j\in\{0,\dots,N(n)-1\}$. By Lemma \ref{lem:smthConn} we can choose $C^1$-smooth paths in $\cP$
        connecting $\gamma(t_j^{(n)})$ and $\gamma(t_{j+1}^{(n)})$ with length
        $\arcsin\bigl(\|\gamma(t_{j+1}^{(n)})-\gamma(t_j^{(n)})\|\bigr)$
        for $j\in\{0,\dots,N(n)\}$. Let $\gamma_n\colon[a,b]\to\cP$ denote the concatenation of these paths for every $n\in\N$.
        Obviously, each $\gamma_n$ is piecewise $C^1$-smooth and has the same endpoints as $\gamma$.
        Since $\gamma_n(t_j^{(n)})=\gamma(t_j^{(n)})$ we have in addition for every $t\in[t_j^{(n)},t_{j+1}^{(n)}]$ the estimate
        \[
            \begin{aligned}
            \norm{\gamma_n(t)-\gamma(t)} &\le \norm{\gamma_n(t)-\gamma_n(t_j^{(n)})} + \norm{\gamma(t_j^{(n)})-\gamma(t)}\\
                    &\le l\Bigl(\gamma_n|_{[t_j^{(n)},t_{j+1}^{(n)}]}\Bigr) + \norm{\gamma(t_j^{(n)})-\gamma(t)}\\
                    &< \arcsin\Bigl(\frac1n\Bigr) + \frac1n\,,
            \end{aligned}
        \]
        and therefore
        \[
         \|\gamma_n(t)-\gamma(t)\| < \arcsin\Bigl(\frac1n\Bigr) + \frac1n
        \]
        for all $t\in[a,b]$, i.e.\ $\gamma_n$ converges uniformly to $\gamma$.

        In order to show that $l(\gamma_n)$ converges to $l(\gamma)$, let $\eps>0$ be arbitrary and take $k\in\N$ such that
        $\frac{l(\gamma)}{k}<\eps$.
        Since $\frac{\arcsin(x)}{x}$ goes to 1 as $x$ approaches zero, there is some $\delta>0$
        such that
        \begin{equation}\label{eq:arcsin}
         \arcsin(x) \le \left(1+\frac1k\right)x
        \end{equation}
        for all $0\le x < \delta$. Due to the lower semicontinuity of the length of paths (cf., \cite[Proposition 1.20]{Bri})
        we can take $N\in\N$ such that
        \begin{equation}\label{eq:upsemicont}
         l(\gamma) \le l(\gamma_n) + \eps
        \end{equation}
        whenever $n\ge N$. We may assume that $\frac{1}{N}<\delta$.
        Taking \eqref{eq:partition} into account, from \eqref{eq:arcsin} and the additivity of the length
        of paths one obtains
        \[
            \begin{aligned}
            l(\gamma_n)
                &= \sum_{j=0}^{N(n)-1} \arcsin\bigl(\|\gamma(t_{j+1}^{(n)})-\gamma(t_j^{(n)})\|\bigr)
                            \le \left(1+\frac1k\right)\sum_{j=0}^{N(n)-1}\|\gamma(t_{j+1}^{(n)})-\gamma(t_j^{(n)})\|\\
                &\le \left(1+\frac1k\right)\cdot l(\gamma)
                \end{aligned}
        \]
        and therefore
        \[
         l(\gamma_n) - l(\gamma) \le \frac{l(\gamma)}{k} < \eps\,.
        \]
        for all $n\ge N$. Together with \eqref{eq:upsemicont} we arrive at
        \[
         |l(\gamma_n) - l(\gamma)| \le \eps
        \]
        whenever $n\ge N$, i.e.\ $l(\gamma_n)$ converges to $l(\gamma)$, which completes the proof.
    \end{proof}

As a consequence, we may restrict our further considerations to
piecewise $C^1$-smooth paths only. The continuous case follows from
that by approximation with piecewise smooth paths. In particular, we
can relax the smoothness hypothesis of Lemma \ref{lem:resPiecSmth}
and obtain the following result, the Arcsine Law for continuous
paths.

\begin{corollary}\label{cor:resCont}
 Let $\gamma\colon[a,b]\to\cP$  be a continuous path. Then
    \begin{equation}\label{eq:resCont}
     \arcsin(\norm{\gamma(b)-\gamma(a)}) \le l(\gamma)\,.
    \end{equation}



\end{corollary}

Recall that given a metric space $(X,d)$, a geodesic path joining
$x$ to $y$ is a map $\gamma$ from a closed interval $[0,l]$ to $X$
such that $\gamma(0)=x$, $\gamma(l)=y$ and
$\rho(\gamma(t),\gamma(s))=\abs{s-t}$ for all $s,t\in[0,l]$. We also
recall that a metric space $(X,d)$ is said to be $r$-geodesic if for
every pair of points $x,y\in X$ with $d(x,y)<r$ there is a geodesic
path joining $x$ to $y$.

The main result of this geometric section  characterizes the local
geodesic behavior of the length space $(\cP,\rho)$. In particular,
we obtain a concrete local representation of the length metric
$\rho$ in terms of the norm of the angle operator.

\begin{theorem}\label{lm}
 The metric space $(\cP,\rho)$ is $\frac{\pi}{2}$-geodesic. Moreover, it is $\rho(P,Q)<\frac{\pi}{2}$ if and only if
    $d(P,Q)=\norm{P-Q}<1$. In that case
    \[
    \rho(P,Q)=\arcsin(\norm{P-Q})\,.
 \]
 In particular,
  $\Ran Q$ is a graph subspace
 with respect to the orthogonal decomposition
 $\cH=\Ran P \oplus \Ran P^\perp$ and
    \[
    \rho(P,Q)=\norm{\Theta}\,,
 \]
 where $\Theta$ is the operator angle between $\Ran Q$ and $\Ran P$.
    \begin{proof}
        Suppose that $P$ and $Q$ are orthogonal projections such that $\rho(P,Q)<\frac\pi 2$. In particular, since
        $\rho(P,Q)$ is finite, this means that there is a continuous path $\gamma$ connecting $P$ and $Q$.
        For any such path we have by Corollary \ref{cor:resCont} that
        \begin{equation}\label{eq:arcsinvslength}
         \arcsin(\norm{P-Q}) \le l(\gamma)\,.
        \end{equation}
        Going to the infimum over connecting paths, we obtain
        \begin{equation}\label{eq:ras}
         \arcsin(\norm{P-Q}) \le \rho(P,Q)\,,
        \end{equation}
        and hence
        \[
         \norm{P-Q}\le \sin(\rho(P,Q))<1\,,
        \]
        due to $\rho(P,Q)<\frac\pi2$\,.

        Conversely, if $\norm{P-Q}<1$, by Lemma \ref{lem:smthConn} there is a $C^1$-smooth geodesic path $\gamma$ connecting
        $P$ and $Q$ of length $l(\gamma)=\arcsin(\norm{P-Q})$ and therefore
        \begin{equation}\label{eq:dva}
         \rho(P,Q) \le l(\gamma) = \arcsin(\norm{P-Q})<\frac\pi2\,.
        \end{equation}
        Thus, $(\cP,\rho)$ is $\frac\pi2$-geodesic, and combining \eqref{eq:ras} and \eqref{eq:dva} proves
        the remaining statement of the theorem.
    \end{proof}%
\end{theorem}


\section{Applications}\label{sec:appl}

Paths  of orthogonal projections  naturally arise  when considering
families of self-adjoint operators  depending  smoothly on a
parameter. Under the additional hypothesis that the self-adjoint
family has a spectrum consisting of two separated parts, the main
problem is  to obtain integral estimates in terms of the relative
strength of the perturbation along the path versus the distance
between the components.
 The upper semicontinuity of the spectrum under a perturbation allows one to obtain efficient estimates on the rotation angle  of the  spectral subspaces, especially in the case
where  the {\it a posteriori} knowledge of the evolution of the separated parts of the  spectra is known.

First, we recall the  concept of an upper semicontinuous family of sets depending on a parameter.

\begin{definition} We say that a family of sets $\{\omega_t\}_{t\in I}$,  with $I$ an interval,
is upper semicontinuous at the point $t\in I$ if   for any $\varepsilon >0$ there exists a $\delta >0$ such that
\begin{equation}\label{eq:addSpecProp}
    \rho(\omega_s, \omega_t)=\sup_{\lambda\in \omega_s}\dist (\lambda,\omega_t)<\varepsilon\quad \text{whenever } \quad |s-t|<\delta,\quad s,t\in I.
    \end{equation}
    The family   $\{\omega_t\}_{t\in I}$, is called upper semicontinuous  on $I$ if it is
    upper semicontinuous at any point $t\in I$.
\end{definition}
Without any loss of generality, we will assume any interval $I$ to contain $0$ throughout this section.

It is well known (see, e.g., \cite[Theorem V.4.10]{Kato}), that
given a $C^1$-smooth path $I\ni t\mapsto B_t$ of self-adjoint
bounded operators, the family of their spectra $\{\spec(B_t)\}_{t\in
I}$ is upper semicontinuous on $I$. Under the additional assumption
that the spectrum of each $B_t$ is separated into two disjoint
components, one can expect the two corresponding families of
spectral components to be upper semicontinuous as well, provided
that they are chosen appropriately. Under these hypotheses, one can
study the variation of the corresponding spectral subspaces under a
variation of the parameter $t\in I$. A natural way of doing that, is
to estimate the deviation  of the corresponding spectral projections
in the length space $(\cP, \rho)$.

As the main application of Theorem \ref{main} we obtain the following result.


\begin{theorem}\label{mainappl}  Assume that $I\ni t \mapsto B_t$ is a $C^1$-smooth path of self-adjoint bounded operators.
Suppose  that the spectrum of each $B_t$ consists of two disjoint spectral components that upper semicontinuously depend on the parameter $t$.
That is, assume that there exist   nonempty closed subsets $\omega_t,\Omega_t\subset\R$ such that for all $t\in I$
\begin{itemize}
\item[(i)] $
    \spec(B_t)=\omega_t\cup\Omega_t$,
\item[(ii)]
    $\dist(\omega_t,\Omega_t)>0$,
\item[(iii)] the families $\{\omega_t\}_{t\in I}$ and  $\{\Omega_t\}_{t\in I}$ are upper semicontinuous on $I$.
\end{itemize}Let
\begin{equation}\label{eq:defSpecProj}
    P_t:= \EE_{B_t}(\omega_t), \quad t\in I,
\end{equation}
denote the spectral projection of the self-adjoint operator  $B_t$ associated with the set
$\omega_t$.

Then
\begin{equation}\label{spes}
\rho(P_t,P_0)=\arcsin (\|P_t-P_0\|)\le \frac{\pi}{2}\int_0^t
\frac{\|\dot B_\tau\|}{\dist (\omega_\tau, \Omega_\tau)}
        \,d\tau
,\quad t\in I.
\end{equation}
\end{theorem}


\begin{proof} We make use of the concept of double operator integrals. Those readers, who
prefer to see a ``standard'' proof   are referred to Appendix
\ref{app:altProof}.


Recall that by the Daletskii-Krein differentiation formula one
obtains the representation
\begin{equation}\label{DK}
\frac{d}{dt}f(B_t)=\int\int\frac{f(\lambda)-f(\mu)}{\lambda-\mu}d\EE_{B_t}
(\lambda)\dot B_td\EE_{B_t}(\mu),
\end{equation}
where $d\EE_{B_t}$ stands for the spectral measure of the self-adjoint operator $B_t$
and $f$ is a $C^\infty$-function on an open interval $(a, b)$ containing the
spectrum of the bounded operator $B_t$.
Under the spectra separation hypothesis one can find an $f\in C^\infty_0([a,b])$ such that
$$
f(\lambda)=\begin{cases}1, & \lambda\in \omega_t,\\
0, & \lambda\in \Omega_t.
\end{cases}
$$
For those $f$'s  one easily concludes that
$$
f(B_t)=\EE_{B_t}(\omega_t)=P_t, \quad t\in I,
$$
and therefore, from \eqref{DK}, one obtains the representation
$$
\dot P_t=\int\int\frac{f(\lambda)-f(\mu)}{\lambda-\mu}d\EE_{B_t}
(\lambda)\dot B_td\EE_{B_t}(\mu),\quad t\in I.
$$
Hence,
\begin{align}
P_t\dot PP_t^\perp &=\int\int\frac{f(\lambda)-f(\mu)}{\lambda-\mu}(P_td\EE_{B_t}
(\lambda)) B_t(d\EE_{B_t}(\mu)P_t^\perp)\nonumber\\
&=\int\int\frac{1}{\lambda-\mu}(P_td\EE_{B_t}
(\lambda))P_t\dot B_tP_t^\perp (d\EE_{B_t}(\mu)P_t^\perp),\quad t\in I.\label{doi}
\end{align}
Since the spectral  measures $ P_td\EE_{B_t}$ and  $d\EE_{B_t}(\mu)P_t^\perp$ are supported by the sets
$\omega_t$ and $\Omega_t$, respectively,  and the sets $\omega_t$ and $\Omega_t$ are separated
with
$$\dist(\omega_t,\Omega_t)>0,
$$
the right hand side of \eqref{doi} can be represented as follows
\begin{equation}\label{trans}
\int\int\frac{1}{\lambda-\mu}(P_td\EE_{B_t}
(\lambda))P_t\dot B_tP_t^\perp (d\EE_{B_t}(\mu)P_t^\perp)=\int_\R e^{isB_t}P_t\dot B_tP_t^\perp  e^{-isB_t}g(s)ds,
\end{equation}
where $g$ denotes any function in $L^1(\R)$, continuous except at zero, such that
$$
\int_\R e^{-is\lambda }g(s)ds=\frac1\lambda \quad \text{whenever} \quad |\lambda|\ge \frac{1}{ \dist(\omega_t,\Omega_t)}.
$$

In particular, one gets the estimate
$$
\|\dot P_t\|=\|P_t\dot P_tP_t^\perp\|\le c \frac{\|\dot P_tB_tP_t^\perp\|}
{\dist(\omega_t,\Omega_t)}\le c \frac{\|\dot B_t\|}
{\dist(\omega_t,\Omega_t)}, \quad t\in I,
$$
where
$$c=\inf \left \{ \norm{g}_{L^1(\R)}\,:\,g\in L^1(\R), \,\widehat g(\lambda)=\frac1\lambda,\, |\lambda|\ge 1\right \},
$$
In fact, see \cite{SS},
$$
c=\frac\pi2,
$$
and hence
one gets the estimate
\begin{equation}\label{better}
\|\dot P_t\|\le \frac\pi2 \frac{\|\dot B_t\|}
{\dist(\omega_t,\Omega_t)}, \quad t\in I.
\end{equation}

 Applying Lemma  \ref{lem:resPiecSmth} completes the proof.
\end{proof}

\begin{remark} We refer to the work of  R. McEachin \cite{Mc} where in fact  it is shown that the norm of the transformer given by the double operator integral \eqref{trans} is $\frac{\pi}{2\dist(\omega_t,\Omega_t)}$ and therefore one cannot expect to get an estimate better than \eqref{better} in general.

However,  if, in addition to the hypotheses of Theorem
\ref{mainappl}, the spectral components $\omega_t$ and $\Omega_t$
are subordinated, i.e.\ $\sup \omega_t > \inf \Omega_t$, or vice
versa,  or if they are annular separated, that is,
 the convex hull of $\omega_t$ lies in the complement to  the set $\Omega_t$   for all $t\in I $, or vice versa, the estimate \eqref{spes} can be strengthened as  follows
\begin{equation}\label{netpi}
\arcsin (\|P_t-P_0\|)\le \int_0^t \frac{\|\dot B_\tau\|}{\dist(\omega_\tau,\Omega_\tau)}
        \,d\tau\,
,\quad t\in I.
\end{equation}
Note, that this estimate is sharp in general (at least in the case of subordinated spectra), as we already know from our previous considerations
in sections \ref{sec:smthPaths} and \ref{sec:geodMetrSp}.
Indeed, for a $C^1$-smooth path $I\ni t\mapsto P_t$ of orthogonal projections take $B_t=P_t$, $\omega_t=\{1\}$ and $\Omega_t=\{0\}$, $t\in I$. Then
it is $\spec(B_t)=\omega_t\cup\Omega_t$ and $\dist(\omega_t,\Omega_t)=1$ for all $t\in I$ and therefore, in this case, \eqref{netpi} coincides with \eqref{eq:resPiecwSmth}, which is
sharp in general.
\end{remark}

The following proposition based on a detailed analysis of one of the
realizations of the Heisenberg commutation relations shows that the
estimate \eqref{spes} in Theorem \ref{mainappl}, being understood in
a somewhat more general context where the consideration of unbounded
operators
   is not excluded, is sharp.

 \begin{proposition}
 Let  $D$  be the differentiation operator with periodic boundary conditions in $L^2(-1,1)$ given by the differential expression
\begin{equation}\label{difex}
D =\frac{d}{dx} \quad\text{on}\quad \Dom (D)=\left \{ f\in W^{2,1}(-1,1)\, | f(-1)=f(1)\right \}.
\end{equation}

Introduce  the isospectral path  $[0,\frac\pi2]\ni t \mapsto B_t$ of unbounded self-adjoint operators
  \begin{equation}\label{iso} B_t=U_t(iD)U_t^*  \quad\text{on}\quad \Dom(B_t)=U_t \Dom(D),
  \end{equation}
  where $U_t$ is the family of unitary operators given by \begin{equation}\label{pathu}
(U_tf)(x)=e^{i2tx}f(x), \quad t\in \left [0, \frac\pi2\right ].
\end{equation}

 Set \begin{equation}\label{spsets}\omega_t=2\pi \Z , \quad \Omega_t =2\pi \Z\setminus\pi \Z,
 \end{equation}
 and  denote by $P_t$ the spectral projection of $B_t$ onto the subspace of ``even harmonics'', that is,
  \begin{equation}\label{proiso}
  P_t=\EE_{B_t}(\omega_t ), \quad t\in \left [0,\frac\pi2\right].
  \end{equation}

  Then
  \begin{equation}\label{fin}
  \arcsin(\|P_t-P_0\|)=\frac{\pi}{2}\int_0^t\frac{\|\overline{\dot B_\tau}\|}{\dist (\omega_\tau, \Omega_\tau)}\, d\tau,
  \quad t\in\left [0,\frac\pi2\right],
  \end{equation}
  where $\overline{\dot B_t}$ denotes the closure of the strong derivative $\dot B_t=\frac{d}{dt} B_t$ of the path  initially defined on
  $$\Dom (\dot B_t)=C^\infty_0(-1,1).$$
 \end{proposition}

\begin{proof} First, one observes that
 \begin{equation}\label{clos}
 \dot B_tf= \frac{d}{dt} B_t f=2U_t[\hat x,D]U^*_tf=-2f \quad\text{ for all } f\in C^\infty_0(-1,1),
  \end{equation}
  where we used the commutation relation
  $$
  [\hat x,D]=I \quad \text{on}\quad C^\infty_0(-1,1)\subset \Dom(\hat xD)\cap \Dom (D\hat x)
  $$
  relating the differentiation operator $D$ and the (bounded) multiplication operator
  $\hat x$ by the independent variable on $L^2(-1,1)$. Thus, the strong derivative $\dot B_t$ is well defined on
  $\Dom(\dot B_t)=C^\infty_0(-1,1)$ and hence
  $$
  \overline{\dot B_t}=-2I_{L^2(-1,1)}.
  $$

  On the other hand, the spectrum of $iD$ consists of simple eigenvalues located at the points of the lattice $2\pi \Z$, so does the spectrum of the isospectral path $B_t$ given by \eqref{iso}. In particular,
  $$
  \dist (\omega_t, \Omega_t)=\pi, \quad t\in\left [0,\frac\pi2\right],
  $$
  and hence
  $$
  \frac{\pi}{2}\int_0^t\frac{\|\overline{\dot B_\tau}\|}{\dist (\omega_\tau, \Omega_\tau)}\, d\tau=
  \frac{\pi}{2}\int_0^t\frac 2\pi\,d \tau=t,\quad t\in\left [0,\frac\pi2\right].
  $$
  To complete the proof of \eqref{fin} it suffices to show that
  $$
  \arcsin(\|P_t-P_0\|)=t,\quad t\in\left [0,\frac\pi2\right].
  $$

  We will prove a slightly more general result that states that the path of the orthogonal projections $\left [0,\frac\pi2\right]\ni t \to P_t$
  is a geodesic.  That is,
 $$
  \arcsin(\|P_t-P_0\|)=\int_0^t\norm{\dot P_\tau}\,d\tau=t,\quad t\in\left [0,\frac\pi2\right].
  $$

  Introduce the notation $P=P_0$. From the definition \eqref{spsets} of the sets $\omega_t$ it follows that $P$ is  the orthogonal projection onto the closure of
$\text{ span}_{k\in \Z}\{ e^{i2kx}\}$,  the space generated by the  ``even harmonics''.
  From \eqref{iso} and \eqref{proiso} it follows that
  \begin{equation}\label{pathp}
P_t=U_tPU_t^*, \quad t\in\left [0,\frac\pi2\right],
\end{equation}
where the family of unitary operators $U_t$ is given by \eqref{pathu}.

First, we prove the inequality
\begin{equation}\label{claim11}
\|P_t-P\|\ge \sin t,\quad t\in\left [0,\frac\pi2\right],
\end{equation}

We proceed as follows.

 One observes that
\begin{align*}
\|(U_tPU_t^*-P)\|&=\|(U_tPU_t^*-P)U_t\|
\\&\ge \|(U_tPU_t^*-P)U_tP^\perp\|
= \|PU_tP^\perp \|,\quad t\in\left [0,\frac\pi2\right].
\end{align*}

The operator $PU_tP^\perp$ can easily be shown to be  unitarily equivalent
(up to a scalar factor) to
the  regularized discrete Hilbert transform $H_p$ in $\ell^2(\Z)$ with
$
p=\frac{2t+\pi}{2\pi}
$,
\begin{equation}\label{fi} PU_tP^\perp\sim-\frac{\sin 2t}{2\pi}{H_p},\quad t\in \left (0,\frac\pi2\right),
\end{equation}
where the symbol $\sim$ denotes a unitary equivalence.

 Recall that by  the  definition the regularized Hilbert transform  $H_p$ is given by the  following convolution operator
$$
(H_p\hat a)_m=\sum_{n\in \Z} \frac{a_n}{m-n+p}, \quad \hat a =\{a_n\}_{n\in \Z}\in \ell^2(\Z), \quad
p\in ( 0,1).
$$

Indeed, to prove \eqref{fi}, take a
 $g\in \Ran P^\perp$  with  the Fourier series
$$
g(x)=\sum_k g_k e^{i(2k+1)\pi x}.
$$
Then  the Fourier series of the function $PU_tP^\perp g$ is given by
\begin{align}
(PU_tP^\perp g)(x)
&=\sum_m \sum_k g_k\frac{1}{2}\left (\int_{-1}^1 e^{i((2(k-m)+1)\pi +2t)\tau }d\tau\right ) e^{2im\pi x}
\label{comp}\\&
=-\sum_{k,m}g_k \frac{ \sin 2t }{(2(k-m)+1)\pi +2t}e^{2im\pi x}\nonumber \\
&=-\frac{\sin 2t}{2\pi}\sum_m (H_p\hat g)_k e^{2im\pi x}, \quad \hat g =\{g_k\}_{k\in \Z}\in \ell^2.\nonumber
\end{align}

Representation \eqref{comp} proves the claim \eqref{fi}. In particular,
\begin{equation}\label{fifi}\|PU_tP^\perp\|=\frac{\sin 2t}{2\pi}\|H_p\|_{\ell^2(\Z)}, \quad  t\in \left (0,\frac\pi2\right).
\end{equation}

Next,  the symbol $h_p$ of the convolution operator $H_p$ can be
computed explicitly and it is given by
$$
h_p(x)=
\frac{\pi}{\sin \pi p} e^{i\pi p(1-x) }= \sum_{m} \frac{e^{im\pi x}}{ m +p}, \quad  x\in   (0,2).
$$
Hence, the norm of $H_p$  in the  space $\ell^2(\Z)$ coincides with the $\ell^\infty$-norm of the symbol $h_p$ and therefore
\begin{equation}\label{ifif}
\|H_p\|_{\ell^2(\Z)}= \sup_{x\in [-1,1]}\left |\frac{\pi e^{i\pi p(1- x)}}{\sin \pi p}\right |=\frac{\pi}{\sin \pi p}=\frac{\pi}{\cos t},\quad t\in \left (0,\frac\pi2\right).
\end{equation}

Combining \eqref{fifi} with \eqref{ifif} yields the  lower bound \eqref{claim11}, that is,
\begin{equation}\label{sin}\|P_t-P\|\ge\|PU_tP^\perp\|= \frac{\sin 2t}{2\pi}\frac{\pi}{\cos t}=\sin t,\quad t\in \left (0,\frac\pi2\right).
\end{equation}

 Our next immediate goal is to prove the opposite inequality
\begin{equation}\label{claim12}
\|P_t-P\|\le \sin t,\quad t\in\left [0,\frac\pi2\right].
\end{equation}

Using the result of Lemma \ref{lem:resPiecSmth}, it is sufficient to prove that
\begin{equation}\label{claim22}
\|\dot P_t\|=1,\quad\,\, \quad t\in \left (0,\frac\pi2\right).
\end{equation}

In order to prove \eqref{claim22},
one  observes that
 $$
\|\dot P_t\|=\|U_t [2\hat x,P]U^*_t\|=2\|[\hat x,P]\|,\quad t\ge 0,
$$
with  $\hat x$ the multiplication operator by the independent variable,
$$(\hat xf)(x)=xf(x), \quad f\in L^2(-1,1).
$$
So,   $\|\dot P_t\|$ does not depend on the parameter $t$.  Therefore,
it remains  to show that
$$
 \|\dot P_0\|=1.
 $$

  Indeed,
 $$
 i^{-1}[2\hat x,P]=iP2\hat xP^\perp-iP^\perp 2\hat xP =iP 2\hat xP^\perp +(iP 2\hat xP^\perp)^*
$$
and hence the commutator $ i^{-1}[2\hat x,P]$ can be represented as the following off-diagonal self-adjoint operator matrix
with respect to the decomposition $L^2(-1,1)=\Ran P\oplus \Ran P^\perp$
$$
i^{-1}[2\hat x,P]=\begin{pmatrix}
0&V\\
V^*&0 \end{pmatrix}.
$$
Here the bounded  operator $V\in \cL(\Ran P^\perp, \Ran P)$, is given by
$$
V=i2P\hat x|_{\Ran P^\perp}.
$$

However, it follows from \eqref{sin} that
$$\|V\|=\left \|\lim_{t\downarrow 0}\frac{P(U_t-I)P^\perp}{it}\right \|=\left \|
\lim_{t\downarrow 0}\frac{PU_tP^\perp}{it}\right \|=\lim_{t\downarrow 0}\frac{\sin t}{t}=1.
$$
Hence,
$$\|[2\hat x,P]\|=\left \| \begin{pmatrix}
0&V\\
V^*&0 \end{pmatrix}\right \|=1.
$$

Combining \eqref{claim11}, \eqref{claim12} and \eqref{claim22}  proves that
the path $\left [0,\frac\pi2\right]\ni t \to P_t$ is geodesic, and hence \eqref{fin} holds.

\end{proof}

\section{New estimates in the subspace perturbation problem}\label{sec:newEstsubPertProb}

The main goal of this section is to apply  Theorem \ref{mainappl} to
the solution of the subspace perturbation problem recently discussed
in \cite{Kos1} and \cite{Kos3}.

Recall that if $A$ and $V$ are self-adjoint bounded operators and
$A$ has a spectral component $\omega$ separated from the rest of the
spectrum $\Omega$, then the spectrum of  $A+V$ still consists of two
separated parts, provided that $\norm{V}$ is small enough.
Due to the upper semicontinuity of the spectrum (cf., e.g., \cite[Theorem V.4.10]{Kato})
this is the case if the (in general sharp) condition $\|V\|<d/2$ with $d=\dist (\omega, \Omega)$
is satisfied.
Moreover,
if the perturbation $V$ is off-diagonal with respect to the
decomposition $\cH=\Ran \EE_A(\omega)\oplus \Ran \EE_A(\Omega)$, in
\cite{Kos1} it is shown, that the
optimal gap nonclosing condition is $\|V\|<\frac{\sqrt{3}}{2}d$ and
this conditions is sharp as well.




It is now a natural question, under what (possibly stronger)
condition on the norm of $V$  the difference of the spectral
projections for $A$ and $A+V$ associated with the corresponding
spectral components is a contraction with the norm  less than 1.


Our first application of Theorem \ref{mainappl} treats   the case of
arbitrary bounded self-adjoint perturbations $V$.

\begin{theorem}\label{thm:diag}
Assume that $A$ and $V$ are bounded self-adjoint operators. Suppose
that the spectrum of $A$ has a part $\omega$ separated from the
remainder of the spectrum $\Omega$ in the sense that
\begin{equation}\spec(A)=\omega \cup \Omega \quad \text{ and }\quad  \dist (\omega, \Omega)=d>0\,.
\end{equation}

If
$$
\|V\|<\frac{\sinh(1)}{e} d\,,
$$
then \begin{equation}\label{bounddiag} \| \EE_A
(\omega)-\EE_{A+V}\left (\cO_{d/2} (\omega)\right )\| \le \sin
\biggl( \frac\pi4 \log \frac{d}{d-2\|V\|}\biggr)<1,
\end{equation}
where $\cO_{d/2} (\omega)$ denotes the open $d/2$-neighborhood of $\omega$.
\end{theorem}


\begin{proof} Introduce the path
$$I=[0,1]\ni t \mapsto B_t=A+tV$$ and set
\begin{equation}\label{eq:defSpecPartsDiag}
    \omega_t:= \spec(B_t)\cap\cO_{ d/2}(\omega)\quad\text{ and } \quad
    \Omega_t:= \spec(B_t)\cap\cO_{d/2}(\Omega)\,,\quad t\in I\,.
\end{equation}
Since
$$
\|V\|<\frac{ \sinh(1)}{e}d<\frac{d}{2}\,,
$$
by \cite[Theorem V.4.10]{Kato} the families $\{\omega_t\}_{t\in I}$ and $\{\Omega_t\}_{t\in I}$
are separated with the distance function $d(t)$ satisfying the estimate
$$
d(t):=\dist (\omega_t, \Omega_t)\ge d-2t \|V\|>0\,, \quad t\in I\,.
$$
Moreover, these families are also upper semicontinuous   on $I$ (cf.,  \cite[Theorem IV.3.16]{Kato} ). Since the path $I\ni t  \mapsto B_t$
is obviously a $C^1$-smooth path  (in fact, it  is real analytic), from Theorem \ref{mainappl} it follows that
\begin{equation}\label{poch}
\arcsin\bigl(\|\EE_A(\omega)-\EE_{B_t}(\omega_t)\|\bigr)\le \frac\pi2\int_0^t \frac{\|V\|}{d-2\tau \|V\|}d\tau\,,
\quad t\in [0,1)\,.
\end{equation}
Observing that
\begin{equation}\label{dob}
\int_0^1 \frac{\|V\|}{d-2\tau \|V\|}d\tau=\frac12\log \left ( \frac{d}{d-2 \|V\|}\right )
\end{equation}
and
$$
\frac12\log \left ( \frac{d}{d-2 \|V\|}\right )<
\frac12\log \left ( \frac{1}{1-2 \frac{\sinh(1)}{e}}\right )=1\,,
$$
from \eqref{poch} (by going to the limit when  $t$ approaches $ 1$) one gets the estimate
\begin{align*}
\|\EE_A(\omega)-\EE_{B_1}(\omega_1)\|
= \sin\biggl(
\frac\pi4
\log \frac{d}{d-2\|V\|}\biggr).
\end{align*}

To complete the proof it remains to observe that
$B_1=A+V
$
and that
$$
\EE_{A+V}\left (\cO_{d/2} (\omega)\right )=\EE_{B_1} (\omega_1)
$$
as it follows from \eqref{eq:defSpecPartsDiag}.
\end{proof}


Our second application of Theorem \ref{mainappl} concerns  the case
of off-diagonal perturbations where the corresponding  spectral
shift is rather specific. In that case, the additional knowledge
about the behavior of the spectral parts from \cite{Kos1} gives rise
to a stronger estimate compared to that in  Theorem \ref{thm:diag}.
\begin{theorem}\label{thm:offdiag}
Assume the hypothesis of Theorem \ref{thm:diag}. Suppose, in addition,  that
$V$ is off-diagonal with respect to the orthogonal decomposition $
\cH=\Ran \EE_A(\omega)\oplus\Ran \EE_A(\Omega), $ that is,
$$
\EE_A(\omega)V\EE_A(\omega)=\EE_A(\Omega)V\EE_A(\Omega)=0\,.
$$
Suppose further that
\begin{equation}\label{sd}
\|V\|< \fs d\,,
\end{equation}
where $\fs$ is the unique root of the equation
\begin{equation}\label{rroo}
\int_0^\fs\frac{d\tau}{2-\sqrt{1+4\tau^2}}=1\,.
\end{equation}

Then
\begin{equation}\label{newoff}
\| \EE_A (\omega)-\EE_{A+V}\left (\cO_{d/2} (\omega)\right )\| \le
\sin \biggl( \frac\pi2
\int_0^{\|V\|/d}\frac{d\tau}{2-\sqrt{1+4\tau^2}} \biggr)<1.
\end{equation}
\end{theorem}


\begin{proof} As in the proof of Theorem \ref{thm:diag}, introduce the path
$I=[0,1]\ni t \mapsto B_t=A+tV$ and the  sets
\begin{equation}\label{eq:defSpecParts}
    \omega_t:= \spec(B_t)\cap\cO_{ d/2}(\omega)\quad\text{ and } \quad
    \Omega_t:= \spec(B_t)\cap\cO_{d/2}(\Omega)\,,\quad t\in I\,.
\end{equation}
Since the improper integral
$$
\int_0^{\frac{\sqrt{3}}{2}}\frac{d\tau}{2-\sqrt{1+4\tau^2}}=\infty
$$
diverges, the root $\fs$ of \eqref{rroo} is well-defined and less than $\frac{\sqrt{3}}{2}$ and hence
\begin{equation}\label{ddff}
\|V\|<\frac{\sqrt{3}}{2}d\,,
\end{equation}
as it follows from \eqref{sd}.

By \cite[Theorem 1.3]{Kos1}, under the condition \eqref{ddff} it is
$$
\spec(B_1)\cap \cO_{d/2}(\omega)=\spec(B_1)\cap \cU_{\delta_V}(\omega)
$$
and also
$$
\spec(B_1)\cap \cO_{d/2}(\Omega)=\spec(B_1)\cap \cU_{\delta_V}(\Omega)\,,
$$
where
$$
\delta_V=\|V\| \tan
\biggl( \frac12 \arctan \frac{2\|V\|}{d}\biggr)\,,
$$
and $U_{\delta_V}(\Delta)$ denotes the closed $\delta_V$-neighborhood of the Borel set $\Delta\subset\R$.
Therefore
\begin{equation}
    \omega_t= \spec(B_t)\cap\cU_{\delta_{tV}}(\omega)\quad\text{ and } \quad
    \Omega_t= \spec(B_t)\cap\cU_{\delta_{tV}}(\Omega)\,,\quad t\in I\,.
\end{equation}

In particular,  the families $\{\omega_t\}_{t\in I}$ and $\{\Omega_t\}_{t\in I}$
are separated with the distance function $d(t)$ satisfying the estimate
\begin{align*}
d(t):&=\dist (\omega_t, \Omega_t)\ge d-2t \|V\|\tan
\biggl( \frac12 \arctan \frac{2t\|V\|}{d}\biggr)\\
&=\left (2-\sqrt{1+4\left (\frac{t\|V\|}{d}\right )^2}\right ) d >0\,, \quad t\in I\,.
\end{align*}

Using a similar argument as in the proof of Theorem \ref{thm:diag} and applying Theorem \ref{mainappl},
one gets the estimate
\begin{align*}
\arcsin\bigl(\|\EE_A(\omega)-\EE_{B_1}(\omega_1)\|\bigr)
&\le
\frac\pi2 \frac{\|V\|}{d}\int_0^1\frac{d\tau}{
2-\sqrt{1+4\left (\frac{\|V\|\tau}{d}\right )^2}}\\
&=
\frac\pi2 \int_0^{\frac{\|V\|}{d}}\frac{d\tau}{
2-\sqrt{1+4\tau^2}}\,.
\end{align*}
By \eqref{sd} and \eqref{rroo} it is
$$
\int_0^{\frac{\|V\|}{d}}\frac{d\tau}{
2-\sqrt{1+4\tau^2}}<\int_0^\fs\frac{d\tau}{
2-\sqrt{1+4\tau^2}}=1\,,
$$
and one arrives at the estimate
$$
\|\EE_A(\omega)-\EE_{A+V}\left (\cO_{d/2}(\omega)\right )\|
=\|\EE_A(\omega)-\EE_{B_1}(\omega_1)\|
\le \sin \biggl(\frac\pi2\int_0^{\frac{\|V\|}{d}}\frac{d\tau}{
2-\sqrt{1+4\tau^2}}\biggr)\,,
$$
which proves \eqref{newoff}.
\end{proof}


\begin{remark} In the situation of Theorem \ref{thm:diag}, the previously known estimate
obtained in \cite{Kos3} under the assumption
$$
\|V\|<\frac{2}{2+\pi}d
$$
has the form
\begin{equation}\label{old}
\| \EE_A (\omega)-\EE_{A+V}\left (\cO_{d/2} (\omega)\right )\|\le \frac\pi2\frac{\|V\|}{d-\|V\|}
\end{equation}
and one can show (see Appendix \ref{app:proofNewVSOld}) that the estimate \eqref{bounddiag} is stronger
than \eqref{old}, i.e.\
\begin{equation}\label{eq:newVSOld}
    \sin \left (
\frac\pi4
\log \frac{d}{d-2\|V\|}\right )<
 \frac\pi2\frac{\|V\|}{d-\|V\|} \qquad \text{whenever}\quad 0<\norm{V}<\frac d2\,.
\end{equation}
In the off-diagonal case of Theorem \ref{thm:offdiag},
the previously known estimate
obtained in \cite{Kos1} under the assumption
\begin{equation}\label{sskk}
\|V\|<\frac{3\pi -\sqrt{\pi^2+32}}{\pi^2-4}d
\end{equation}
has the form
\begin{equation}\label{oldoff}
\| \EE_A (\omega)-\EE_{A+V}\left (\cO_{d/2} (\omega)\right )\|\le \frac\pi2\frac{\|V\|}{d-\|V\|\tan \left (\frac12\arctan \frac{2\|V\|}{d}\right )}\,.
\end{equation}
Recall that the critical constant $c_\pi=\frac{3\pi -\sqrt{\pi^2+32}}{\pi^2-4}$ in \eqref{sskk} was chosen to be the only positive root of the equation
$$
\frac\pi2\frac{x}{1-x\tan \left (\frac12\arctan 2x\right )}=1
$$
and therefore in the critical case  $\|V\|=c_\pi d$ the right hand side of \eqref{oldoff} turns out to be 1 which means that the bound \eqref{oldoff} is not informative for the range of perturbations $\|V\|$ such that $\|V\|\ge c_\pi d$.

Note, that the identity
\[
 \frac{1}{2-\sqrt{1+4\tau^2}} = \frac{1}{1-2\tau\tan\bigl(\frac{1}{2}\arctan 2\tau\bigr)}
\]
holds for all $0\le\tau<\frac{\sqrt{3}}{2}$. Combined with the inequality
\begin{equation}\label{notvery}
\sin \left ( \frac\pi2 \int_0^t \frac{d\tau}{1-{2\tau} \tan \left (\frac12\arctan 2\tau\right )}\right )<
\frac\pi2
\frac{t}{1-{t} \tan \left (\frac12\arctan 2t\right )}\,,
\end{equation}
$$0\le t<\frac{\sqrt{3}}{2}\,,
$$
which is proven in Appendix \ref{app:proofnotvery}, this means that
  the right hand side of \eqref{newoff} is less than the right hand side of
\eqref{oldoff}
and therefore
the bound \eqref{newoff} is stronger than the previously known estimate \eqref{oldoff}.
In particular, since the critical constant $c_\pi$ was defined so that the right hand side of \eqref{oldoff} equals 1 for $\norm{V}=c_\pi d$,
it follows immediately from \eqref{notvery} that $c_\pi<\fs$. Numerical calculations suggest that the exact value of $\fs$ satisfies the
two-sided estimate
\[
    0{.}67598931 < \fs < 0{.}67598932\,.
\]
\end{remark}

\begin{appendix}

\section{Proof of inequality \eqref{eq:newVSOld}}\label{app:proofNewVSOld}
We write $\norm{V}=\alpha d$ and substitute $x=\frac{\alpha}{1-\alpha}$, $0<x<1$. With $\frac{1}{1-2\alpha}=\frac{1+x}{1-x}$
inequality \eqref{eq:newVSOld} becomes
\begin{equation}\label{eq:newVSOld2}
 \sin\biggl(\frac{\pi}{4}\log \frac{1+x}{1-x}\biggr) < \frac{\pi}{2} x\,,\quad 0<x<1\,.
\end{equation}
Since the left-hand side of \eqref{eq:newVSOld2} is not greater than 1, we may assume $x\le \frac2\pi$. In that case,
\eqref{eq:newVSOld2} can be rewritten as
\begin{equation}\label{eq:arcsinOldNew}
 \frac{\pi}{4} \log \frac{1+x}{1-x} < \arcsin\Bigl(\frac{\pi}{2} x\Bigr)\,,\quad 0<x\le\frac2\pi\,.
\end{equation}
It suffices to show that the corresponding inequality holds for the derivatives of both sides. Differentiating the left-hand side
gives
\begin{equation*}
 \frac{d}{dx} \frac\pi4\log \frac{1+x}{1-x} = \frac\pi2\cdot \frac{1}{(1-x^2)}
\end{equation*}
and differentiating the right-hand side gives
\begin{equation*}
 \frac{d}{dx} \arcsin\Bigl(\frac{\pi}{2}x\Bigr) = \frac{\pi}{\sqrt{4-\pi^2x^2}}\,.
\end{equation*}
Therefore, we have to show that
\begin{equation}\label{eq:diffIneq}
 \frac\pi2\cdot \frac{1}{(1-x^2)} < \frac{\pi}{\sqrt{4-\pi^2x^2}}
\end{equation}
holds for all $0<x\le\frac2\pi$. Taking the square, we can rewrite \eqref{eq:diffIneq} as
\[
 4-\pi^2x^2 < 4 (1-x^2)^2
\]
which is equivalent to
\[
 0 < (\pi^2 - 2)x^2 + x^4\,,\quad 0 < x \le \frac2\pi\,.
\]
Since $\pi^2 > 2$, this is obviously true, so \eqref{eq:diffIneq} holds for all $0<x\le\frac2\pi$, which proves
\eqref{eq:newVSOld}.

\section{ Proof of  inequality \eqref{notvery}}\label{app:proofnotvery}

First, we remark that
$$
1-{2x} \tan \left (\frac12\arctan 2x\right )=2-\sqrt{1+4x^2}
$$
and that
$$
1-{x} \tan \left (\frac12\arctan 2x\right )=\frac32-\frac{\sqrt{1+4x^2}}{2}
$$
for $0\le x<\frac{\sqrt{3}}{2}$,
and thus the inequality \eqref{notvery} can be rewritten as
\begin{equation}\label{notvery1}
\sin \biggl( \frac\pi2 \int_0^t \frac{d\tau}{2-\sqrt{1+4\tau^2}}\biggr)<
\frac\pi2
\frac{2t}{3-\sqrt{1+4t^2}}\,.
\end{equation}

It is sufficient to prove the corresponding inequality  for the derivatives that, after elementary  computations, can be written as
\begin{equation}\label{iter} \cos \left ( \frac\pi2 \int_0^t \frac{d\tau}{2-\sqrt{1+4\tau^2}}\right )<2\left [
\frac{2-\sqrt{1+4t^2}} { 3-\sqrt{1+4t^2}}+\frac{4t^2( 2-\sqrt{1+4t^2})} {( 3-\sqrt{1+4t^2})^2\sqrt{1+4t^2}}
\right ].\end{equation}

After the  change of variables $x=\sqrt{1+4t^2}$, so that $t^2=\frac{x^2-1}{4}$,
the desired estimate may be rewritten as
$$ \cos \biggl( \frac\pi2 \int_0^{\frac{\sqrt{x^2-1}}{2}} \frac{d\tau}{2-\sqrt{1+4\tau^2}}\biggr)<
F(x)\,, \quad 1<x<2\,,
$$ where the function $F$ is given by
$$
F(x)=\frac{2(3x-1)(2-x)} {( 3-x)^2x}\,.
$$
Denote by $\fx$ the first root of the equation
$$
\cos \left ( \frac\pi2 \int_0^{\frac{\sqrt{x^2-1}}{2}} \frac{d\tau}{2-\sqrt{1+4\tau^2}}\right )=
F(x)
$$
that is greater than $1$.

An elementary analysis shows that the  equation
$1=F(x)
$
has three roots
$$
x_1=-\frac{\sqrt{17}+1}{2}\,, \quad x_2=1 \quad \text{ and }\quad x_3=\frac{\sqrt{17}-1}{2}\,,
$$
and that
$$
1<F(x) \quad \text{ on } \quad  (1,x_3)\,.
$$
Therefore,
$$
\fx>x_3=\frac{\sqrt{17}-1}{2}
$$
and thus, we have proven the inequality
\begin{equation}\label{notvery11}
\sin \biggl( \frac\pi2 \int_0^t \frac{d\tau}{2-\sqrt{1+4\tau^2}}\biggr)<
\frac\pi2
\frac{2t}{3-\sqrt{1+4t^2}} \quad \text{on the interval }\quad (0, \ft)\,,
\end{equation}
where $\ft$ is given by
$$
\ft=\frac{\sqrt{x_3^2-1}}{2}=\frac{\sqrt{(\sqrt{17}-1)^2-4}}{4}\approx
0.599...\,\,.
$$

Next, we show that the right  hand side of \eqref{notvery11}  at the point $t=\ft$ is greater than $1$. Indeed,
\begin{align}
\frac\pi2
\frac{2\ft}{3-\sqrt{1+4\ft^2}}&=\frac\pi2\frac{2\frac{\sqrt{(\sqrt{17}-1)^2-4}}{4}}{3-\frac{\sqrt{17}-1}{2}}
>\pi \frac{\frac{\sqrt{(\sqrt{16}-1)^2-4}}{4}}{3-\frac{\sqrt{17}-1}{2}}
=\pi \frac{\sqrt{5}}{14-2\sqrt{17}}
\label{vvbb} \\&>\pi\frac{\sqrt{5}}{5}=\frac{\pi}{\sqrt{5}}>1,
\nonumber\end{align}
where we used the obvious inequality
$2\sqrt{17}<9.
$


From that it follows, that the right hand side of \eqref{notvery} is greater than $1$ for $t\ge \ft$ and one concludes that
\begin{equation}\label{notvery3}
\sin \biggl( \frac\pi2 \int_0^t \frac{d\tau}{2-\sqrt{1+4\tau^2}}\biggr)<
\frac\pi2
\frac{2t}{3-\sqrt{1+4t^2}} \quad \text{for all }\ t\in\left  (0, \frac{\sqrt{3}}{2}\right )\,.
\end{equation}
We remark that at $t=\frac{\sqrt{3}}{2}$  the right hand side of \eqref{notvery3} blows up.

\section{Alternative Proof of Theorem \ref{mainappl}}\label{app:altProof}
In this appendix we present an alternative proof of Theorem
\ref{mainappl}.

 We make the following preparations.
\begin{lemma}\label{lem:specDistCont}
    Under the assumptions of Theorem \ref{mainappl} we define $d\colon I\times I\to\R_0^+$ by
    \[
        d(t,s):=\dist(\omega_t,\Omega_s)\,.
    \]
    Then for every $t\in I$ one has
    \[
        \lim_{s\to t}d(t,s)=\lim_{s\to t}d(s,t)=d(t,t)\,.
    \]
    Furthermore, the function given by $t\mapsto d(t,t)$, $t\in I$, is continuous.
    \begin{proof}
        Let $t\in I$ and let $0<\varepsilon<\frac{1}{4}d(t,t)$ be arbitrary. Since the families $\{\omega_s\}_{s\in I}$ and
        $\{\Omega_s\}_{s\in I}$ are upper semicontinuous, we can choose $\delta>0$ such that
        \[
         \omega_s\subset \cO_\eps(\omega_t)\,,\quad \Omega_s\subset \cO_\eps(\Omega_t)\,,\quad s\in I\,,\ \abs{s-t}<\delta\,,
        \]
        as well as
        \[
         \omega_t\subset \cO_\eps(\omega_s)\,,\quad \Omega_t\subset \cO_\eps(\Omega_s)\,,\quad s\in I\,,\ \abs{s-t}<\delta\,,
        \]
        where $\cO_\eps(\Delta)$ denotes the open $\eps$-neighborhood of $\Delta\subset\R$.
        From that one obtains
        \[
            d(t,t)-2\varepsilon \le d(s,s)\le d(t,t)+2\varepsilon
        \]
        and
        \[
            d(t,t)-\varepsilon \le d(t,s)\le d(t,s)+\varepsilon
        \]
        for all $s\in I$ such that $\abs{s-t}<\delta$. The same is true for $d(s,t)$ instead
        of $d(t,s)$, which completes the proof.
    \end{proof}%
\end{lemma}

\begin{proposition}[\cite{Mc}]\label{mcen} Let $A$ and $B$ be bounded self-adjoint operators and $\omega$ and $\Omega$ two Borel sets on the real line. Then
$$
\dist (\omega, \Omega)\|\EE_A (\omega)\EE_B(\Omega)\|\le \frac\pi2 \|A-B\|.
$$ Moreover, if the convex hull of the set $\omega$ does not intersect the set $\Omega$, or vice versa,  then
$$
\dist (\omega, \Omega)\|\EE_A  ( \omega)\EE_B(\Omega)\|\le  \|A-B\|.
$$
\end{proposition}



Now we are able to prove Theorem \ref{mainappl}.


\begin{proof}[Proof of Theorem \ref{mainappl}] By Proposition \ref{mcen},
\begin{equation}\label{p1}
\dist (\omega_t,\Omega_s)\|P_tP_s^\perp\|\le \frac\pi2 \|B_t-B_s\|, \quad s,t\in I,
\end{equation}
and
\begin{equation}\label{p2}
\dist (\omega_s,\Omega_t)\|P_t^\perp P_s\|\le \frac\pi2 \|B_t-B_s\|, \quad s,t\in I.
\end{equation}
Since
$$
\|P_s-P_t\|=\max\left \{\|P_tP_s^\perp\| , \|P_t^\perp P_s^\perp\|
\right \},
$$
from \eqref{p1} and \eqref{p2} it follows  that
$$
\min\left \{
\dist (\omega_t,\Omega_s), \dist (\omega_s,\Omega_t)
\right \} \|P_s-P_t\|\le \frac\pi2 \|B_s-B_t\|,\quad s,t \in I.
$$
Dividing both sides of this inequality by $\abs{s-t}$  and  letting  $s$ approach $t$,
one obtains the bound
$$
\dist(\omega_t,\Omega_t)\|\dot P_t\|\le \frac\pi2\|\dot B_t\|, \quad t\in I\,,
$$
where we have used Lemma \ref{lem:specDistCont} and the smoothness of the path $I\ni t\mapsto P_t$
(cf.\ Appendix \ref{app:proofSmth}).

Since $\dist(\omega_t,\Omega_t)>0$ for all $t\in I$ by hypothesis,
one obtains that
$$
\|\dot P_t\|\le \frac\pi2 \frac{\|\dot B_t\|}{\dist(\omega_t,\Omega_t)}, \quad t\in I,
$$ and then
applying Lemma \ref{lem:resPiecSmth} completes the proof.
\end{proof}

\section{Proof of the smoothness of the spectral projections}\label{app:proofSmth}

The proof of the smoothness of the path of projections $P_t$
required for the alternative proof  of Theorem \ref{mainappl} in
Appendix C is essentially the same as the one presented in
\cite[Theorem II.5.4]{Kato} for the continuous case.

\begin{lemma}\label{lem:specProjSmooth}
    Under the assumptions of Theorem \ref{mainappl}, $I\ni t\mapsto P_t$ is a $C^1$-smooth path.
    \begin{proof}
        Let $t\in I$ and $\varepsilon=\frac{1}{4}\dist(\omega_t,\Omega_t)>0$.
        Due to the fact, that the families $\{\omega_s\}_{s\in I}$ and $\{\Omega_s\}_{s\in I}$ are upper semicontinuous,
        there is a $\delta>0$ such that
        \begin{equation}\label{eq:specPropII}
            \omega_s\subset \cO_{\varepsilon/2}(\omega_t)\ \text{ and }\ \Omega_s\subset\cO_{\varepsilon/2}
            (\Omega_t)\quad \text{ for all}\quad s\in I\,,\ \abs{s-t}<\delta\,,
        \end{equation}
        where $\cO_{\eps}(\Delta)$ denotes the open $\eps$-neighborhood of $\Delta\subset\R$.

        In particular, $\cO_\varepsilon(\omega_t)\setminus \cO_{\varepsilon/2}(\omega_t)$ lies in the resolvent set of $B_s$
        for all $s\in I\,,\ \abs{s-t}<\delta$. Therefore, there exists a finite number of rectifiable, simple closed
        positive orientated curves belonging to $\C\setminus\spec(B_s)$ for all $s\in I\,, \abs{s-t}<\delta$,
        such that $\omega_s$ is contained in the union of their interiors and $\Omega_s$ lies in the
        union of their exteriors. Let $\Gamma$ denote the union of these curves.
        As in \cite[(III.6.19)]{Kato}, $P_s$ has the representation
        \begin{equation*}
            P_s = \frac{1}{2\pi \ii} \int_\Gamma R_s(\zeta) \,d\zeta\,,\quad
            R_s(\zeta) := \bigl(\zeta I_\cH - B_s\bigr)^{-1}\,,\quad s\in I\,,\ \abs{s-t}<\delta\,.
        \end{equation*}
        Since $B_s-B_t = (\zeta I_\cH - B_t)-(\zeta I_\cH - B_s)$, it is
        \[
            R_s(\zeta)-R_t(\zeta) = R_t(\zeta)\bigl(B_s - B_t\bigr)R_s(\zeta)\,,\quad s\in I\,,\ \abs{s-t}<\delta\,.
        \]
        Furthermore, for all $\zeta\in\Gamma\subset\C\setminus\spec(B_s)$ the equation
        \begin{equation}\label{eq:NormResolvent}
            \|{R_s(\zeta)}\|=\frac{1}{\dist(\zeta,\spec(B_s))}
        \end{equation}
        holds (cf. \cite[(V.3.16)]{Kato}). Due to \eqref{eq:specPropII}, this implies that
        $\|{R_s(\zeta)}\|$ is uniformly bounded for $\zeta\in\Gamma$ and $s\in I\,,\ \abs{s-t}<\delta$, from which
        one concludes, that $\frac{R_s(\zeta)-R_t(\zeta)}{s-t}$ converges uniformly
        to $R_t(\zeta)\dot B_t R_t(\zeta)$ for $\zeta\in\Gamma$ as $s$ goes to $t$.
        This shows that
        \[
            \dot P_t = \lim_{s\to t}\frac{P_s-P_t}{s-t}=\frac{1}{2\pi\ii}\int_\Gamma
            R_t(\zeta)\dot B_t R_t(\zeta)\,d\zeta
        \]
        exists. By a similar argument, one concludes that $I\ni t\mapsto \dot P_t$ is continuous and, therefore, $I\ni t\mapsto P_t$ is
        $C^1$-smooth.
    \end{proof}%
\end{lemma}
\end{appendix}


\begin{thebibliography}{[10]}
    \bibitem{Ach} N. I. Achiezer, I. M. Glasmann, \emph{Theory of Linear Operators in Hilbert Space}, Dover Publications, New York, 1993.
    \bibitem{Bri} M.~R.~Bridson, A.~Haefliger, \emph{Metric spaces of non-positive curvature}, Springer, Berlin, 1999.

\bibitem{Halmos} P.~R.~Halmos, {\it Two subspaces}, Trans. Amer. Math. Soc. {\bf 144} (1969),
381--389.


    \bibitem{Kato} T.~Kato, \textit{Perturbation Theory for Linear Operators}, Springer-Verlag, Berlin, 1966.
    \bibitem{Kos1} V. Kostrykin, K. A. Makarov, A. K. Motovilov, \emph{Perturbation of spectra and spectral subspaces},
            \textbf{359} (2007), 77-89.
    \bibitem{Kos2} V. Kostrykin, K. A. Makarov, A. K. Motovilov, \emph{Existence and Uniqueness of Solutions to the Operator Riccati Equation. A Geometric
            Approach}, Contemporary Mathematics Vol. 327, Amer. Math. Soc., 2003, p. 181 - 198.
    \bibitem{Kos3} V. Kostrykin, K. A. Makarov, A. K. Motovilov, \emph{On a subspace perturbation problem}, Proc.\ Amer.\ Math.\ Soc.\
            \textbf{131} (2003), 3469-3476.

    \bibitem{Mc} R. McEachin, \emph{Closing the gap in a subspace perturbation bound}, Linear Algebra Appl. {bf 180} (1993), 7--15.
    \bibitem{SS} B. Sz.-Nagy, \emph{\"Uber die Ungleichnung von H. Bohr.}, Math. Nachr. {\bf 9} (1953), 255--259.

\end{thebibliography}
\end{document}